\documentclass[11pt,reqno]{amsart}

\usepackage{fullpage}

\usepackage{amsmath}
\usepackage{amssymb}
\usepackage[T1]{fontenc}
\usepackage{mathrsfs}
\usepackage{enumitem}
\usepackage{stmaryrd}
\usepackage{hyperref}
\usepackage{url}
\usepackage{fancyhdr}
\usepackage{mathtools}
\usepackage{comment}
\usepackage{units}
\usepackage{datetime}
\longdate

\usepackage{color,soul}
\usepackage{tikz}
\usepackage{subfigure}

\newtheorem{theorem}{Theorem}[section]

\newtheorem{proposition}[theorem]{Proposition}
\newtheorem{corollary}[theorem]{Corollary}
\newtheorem{definition}{Definition}

\theoremstyle{definition}

\theoremstyle{remark}
\newtheorem{remark}{\textsc{Remark}}
\newtheorem*{claim*}{\textsc{Claim}}

\renewcommand\epsilon{\varepsilon}
\renewcommand\ss{\mathfrak{s}}
\renewcommand\tt{\mathfrak{t}}

\providecommand\llb{\llbracket}
\providecommand\rrb{\rrbracket}

\providecommand\gonc{Gon\v{c}arov{ }}

\providecommand{\dd}{\mathfrak{d}}

\providecommand{\FFc}{\mathcal{F}}

\providecommand{\HHc}{\mathcal{H}}
\providecommand{\KK}{\mathbb{K}}

\providecommand{\NNb}{\mathbf{N}}

\providecommand{\OOc}{\mathcal{O}}
\providecommand{\PPc}{\mathcal{P}}

\providecommand{\TTc}{\mathcal{T}}
\providecommand{\WWc}{\mathcal{W}}

\providecommand{\ZZc}{\mathcal{Z}}
\renewcommand{\tt}{t}

\providecommand\llb{\llbracket}
\providecommand\rrb{\rrbracket}

\providecommand{\imag}{{\rm Im}}

\providecommand{\uimag}{\underline{\imag}}

\newcommand\iter[2]{#1^{#2}}

\newcommand{\fixed}[2][1]{%
  \begingroup
  \spaceskip=#1\fontdimen2\font minus \fontdimen4\font
  \xspaceskip=0pt\relax
  #2%
  \endgroup
}

\makeatletter
\def\moverlay{\mathpalette\mov@rlay}
\def\mov@rlay#1#2{\leavevmode\vtop{%
   \baselineskip\z@skip \lineskiplimit-\maxdimen
   \ialign{\hfil$\m@th#1##$\hfil\cr#2\crcr}}}
\newcommand{\charfusion}[3][\mathord]{
    #1{\ifx#1\mathop\vphantom{#2}\fi
        \mathpalette\mov@rlay{#2\cr#3}
      }
    \ifx#1\mathop\expandafter\displaylimits\fi}
\makeatother

\newcommand{\cupdot}{\charfusion[\mathbin]{\cup}{\cdot}}
\newcommand{\bigcupdot}{\charfusion[\mathop]{\bigcup}{\cdot}}

\hypersetup{
    pdftitle={Generalized Goncarov polynomials},
    pdfauthor={Lorentz, Tringali, and Yan},
    pdfmenubar=false,
    pdffitwindow=true,
    pdfstartview=FitH,
    colorlinks=true,
    linkcolor=blue,
    citecolor=green,
    urlcolor=cyan
}

\begin{document}
\title{Generalized Gon\v{c}arov polynomials}

\author{Rudolph Lorentz}
\address{Department of Mathematics, Texas A\&M University at Qatar \\ PO Box 23874 Doha, Qatar}
\email{rudolph.lorentz@qatar.tamu.edu}
\urladdr{http://people.qatar.tamu.edu/rudolph.lorentz/} 

\author{Salvatore Tringali}
\address{Department of Mathematics, Texas A\&M University at Qatar \\ PO Box 23874 Doha, Qatar}
\email{salvo.tringali@gmail.com}
\urladdr{http://www.math.polytechnique.fr/~tringali/}

\author{Catherine H. Yan}
\address{Department of Mathematics, Texas A\&M University \\ 77845-3368 College Station (TX), US}
\email{cyan@math.tamu.edu}
\urladdr{http://www.math.tamu.edu/~cyan/}


%
%

\begin{abstract}
We introduce the sequence of generalized Gon\v{c}arov polynomials, which is  a basis for the solutions to
the Gon\v{c}arov interpolation problem with respect to a delta operator.  Explicitly,
a generalized Gon\v{c}arov basis is  a sequence $(t_n(x))_{n \ge 0}$ of polynomials defined by the biorthogonality
relation $\varepsilon_{z_i}(  \mathfrak d^{i}(t_n(x))) = n! \;\! \delta_{i,n}$ for all $i,n \in \mathbf N$,
where $\mathfrak d$ is a delta operator,  $\mathcal Z = (z_i)_{i \ge 0}$ a sequence of scalars,
and $\varepsilon_{z_i}$ the evaluation at $z_i$.
We present algebraic and analytic  properties of generalized  Gon\v{c}arov polynomials
and   show that such polynomial sequences provide a natural algebraic tool for enumerating
combinatorial structures with a linear constraint on their order statistics.
\end{abstract}

\dedicatory{Dedicated to Ronald Graham on the occasion
     of his $80$th birthday}

\clearpage
\maketitle
\thispagestyle{empty}


\section{Introduction}
\label{sec:intro}

This paper is a work combining three areas: interpolation theory, finite operator calculus, and combinatorial enumeration.
Lying in the center is a sequence of polynomials, the generalized \gonc polynomials, that arose from
the \gonc Interpolation problem in numerical analysis.

The classical \gonc Interpolation problem  is a special case of Hermite-like interpolation.
It asks for  a polynomial $f(x)$ of degree $n$ such that the $i$th derivative of $f(x)$ at a given point $a_i$ has value
$b_i$, for $i=0,1, \ldots, n$.
The problem was introduced by \gonc \cite{Gonc30, Goncarov} and Whittaker \cite{Whittaker}, and the solution is obtained by
taking linear combinations of the  (classical) \gonc polynomials, or
the Abel-\gonc polynomials, which have been studied extensively by analysts, due to their considerable
 significance in the interpolation
theory of smooth and analytic functions, see for instance \cite{Gonc30, Lev, FrSh, Has} and references therein.

Surprisingly, \gonc polynomials also play an important role in combinatorics due to their close relations to
parking functions,
A parking function is a sequence $(a_1, a_2, \ldots, a_n)$ of positive integers
such that for every $i=1, 2, \ldots, n$, there are at least $i$ terms that are less than or equal to $i$.
Parking functions are one of the most fundamental objects in combinatorics and are related to many different
structures, for example, labeled trees and graphs, linear probing in computer algorithms, hyperplane arrangements,
non-crossing partitions, and diagonal harmonics and representation theory, to name a few.   See, for example, \cite{Yan14} for
a comprehensive survey on parking functions.
It is shown by Kung and Yan \cite{KuYan} that \gonc polynomials are the natural basis of polynomials for
working with parking functions, and the enumeration  of parking functions and their generalizations can be obtained  using \gonc polynomials.
Khare, Lorentz and Yan \cite{KLY14} investigated  a multivariate \gonc interpolation problem  and defined  sequences of
multivariate \gonc polynomials,
which are solutions to the interpolation problem and   enumerate $k$-tuples of
integer sequences whose order statistics are bounded by certain weight function along lattice paths in
$\NNb^k$.

In \cite{RoKaOd73}, Rota, Kahaner and Odlyzko introduced a unified theory of special polynomials by exploiting to the
hilt the duality between $x$ and $d/dx$. The main technique is  a rigorous version of  symbolic calculus, also called
\emph{finite operator calculus}, since it has  an  emphasis on operator methods.
This algebraic theory  is particularly useful in dealing with
polynomial sequences of binomial type, which occur in a large variety of
combinatorial problems when one wants to enumerate objects that can be pieced together out of small, disjoint objects.
Each polynomial sequence of binomial type can be characterized by a linear operator called
\emph{delta operator}, which possesses many properties of the differential operator.  A few basic principles of delta operators
lead to a series of expansion and isomorphism theorems on families of special polynomials,
which in turn lead to new  identities and solutions to combinatorial problems.

Inspired by the rich theory on delta operators,
we  extend the \gonc interpolation problem by replacing the differential operator with a delta operator and
consider the following interpolation.
\begin{quote}
{\bf Generalized Gon\v{c}arov Interpolation.} Given
two sequences $z_0, z_1, \ldots, z_{n}$ and $b_0, b_1, \ldots, b_{n}$ of real or complex numbers and a delta operator $\dd$,
find a (complex) polynomial $p(x)$
of degree $n$ such that
\begin{eqnarray}\label{eq:goncarov-interpolation}
  \varepsilon_{z_i} \iter{\dd}{i} (p(x)) = b_i \qquad \text{ for each } i=0, 1, \ldots, n.
\end{eqnarray}
\end{quote}
The solution of this problem rises the generalized \gonc polynomial.
When $\dd=D$, we recover  the classical \gonc polynomials.
Generalized \gonc polynomials enjoy many nice algebraic properties and
carry   a combinatorial  interpretation  that combines the ideas of   binomial enumeration and  order statistics. Roughly speaking,
if each combinatorial object
is associated with  a sequence of numbers and we rearrange those numbers in
non-decreasing order, then the generalized \gonc polynomial counts those objects of which  the non-decreasing rearrangements are
bounded by a predetermined sequence.
Such structures give  a new generalization  of the classical parking functions.

The main objective of this paper is to present the algebraic and combinatorial properties of generalized
\gonc polynomials.
The paper is organized as follows. In Section 2 we recall the basic theory of delta operators, polynomial sequences of binomial type, and
polynomial sequences biorthogonal to a sequence of linear operators. Using this theory we introduce the sequence of generalized \gonc
polynomials $t_n(x; \dd,  \ZZc)$ associated with a delta operator $\dd$ and a grid $\ZZc$.  In the subsequent Sections 3-5, we discuss
the algebraic properties, present  a combinatorial formula for $t_n(x; \dd, \ZZc)$,
and describe the combinatorial interpretation
by counting  reluctant functions, a kind of combinatorial structures arising from the study of binomial enumeration, with
 constraints on the order statistics.
Many examples are given in Section 6. We finish the paper with some further remarks in the
last section.

\section{Delta Operators and Generalized \texorpdfstring{Gon\v{c}arov}{Goncarov} polynomials}
\label{sec:generalize_goncarov_polys}

\subsection{Delta operator and basic polynomials } \

We start by recalling the basic theory  of delta operators and their associated sequence of basic polynomials,
as developed by Mullin and Rota  \cite{RoKaOd73}.

Consider the vector space $\KK[x]$ of all polynomials in the variable $x$ over a field
$\KK$ of  characteristic zero.  For each $a \in \KK$,
let $E_a$ denote the shift operator $\KK[x] \to \KK[x]: f(x) \mapsto f(x+a)$.
A linear operator $\ss: \KK[x] \to \KK[x]$ is called \textit{shift-invariant} if
$\ss E_a = E_a \fixed[0.25]{ \text{ }} \ss$ for all $a \in \KK$.
\begin{definition}
 A  \emph{delta operator} $\dd$ is a  shift-invariant  operator satisfying $\dd(x) =a$ for
some nonzero constant $a$.
\end{definition}
Delta operators possess many of the properties of the differentiation operator $D$. For
example, $\deg(\dd(f))=\deg(f)-1$ for any $f \in \KK[x]$ and $\dd(a)=0$ for every constant $a$.

We say that a shift-invariant operator $\ss$ is \emph{invertible} if  $\ss(1) \neq 0$. Note that
delta operators are not invertible.

\begin{definition}
 Let $\dd$ be a delta operator. A polynomial sequence $(p_n(x))_{n \ge 0}$ is called the \emph{sequence of basic polynomials},
 or the \emph{basic sequence},
of $\dd$ if
\begin{enumerate}[label={\rm (\arabic{*})}]
 \item $p_0(x)=1$;
 \item $p_n(0)=0$ whenever $n \geq 1$;
 \item $\dd (p_n(x)) = np_{n-1}(x)$.
\end{enumerate}
\end{definition}

Every delta operator has a unique sequence of basic polynomials, which is a sequence of binomial type, i.e., satisfies
\begin{eqnarray} \label{eq:binomial}
 p_n(x+y) =\sum_{k \geq 0} \binom{n}{k} p_k(x) p_{n-k}(x) \qquad \text{for all } n.
\end{eqnarray}
Conversely, every sequence of polynomials of binomial type is the basic sequence for some delta operator.

Let $\ss$ be a shift-invariant operator, and $\dd$ a delta operator with basic sequence $p_n(x)$. Then $\ss$ can be expanded
as a formal power series of $\dd$, as
\begin{eqnarray} \label{eq:delta-expansion}
 \ss = \sum_{k \geq 0} \frac{a_k}{k!} \dd^k
\end{eqnarray}
with $a_k = \varepsilon_0( \ss(p_k(x))$. We will say that the formal power series
$f(t)=  \sum_{k \geq 0} \frac{ a_k}{k!} t^k$   is the $\dd$-indicator of $\ss$.
In fact, there exists an isomorphism from the ring $\KK \llb t \rrb$ of formal power series in the
variable $t$ over $\KK$ onto the ring
$\Sigma$ of shift-invariant operators, which carries
\begin{eqnarray}\label{eq:delta-FPS}
 f(t) = \sum_{k \geq 0} \frac{a_k}{k!} t^k \qquad \text{into} \qquad \sum_{k \geq 0}  \frac{a_k}{k!} \dd^k.
\end{eqnarray}
Under this isomorphism, a shift-invariant operator $\ss$ is invertible if and only if its $\dd$-indicator
$f(t)$ satisfies
$f(0)\neq 0$,  and $\ss$ is a delta operator if and only if $f(0)=0$ and $f'(0)\neq 0$.

Another result we will need is the generating function for the sequence of basic polynomials $\{p_n(x)\}$
associated to a delta
operator $\dd$. Let $q(t)$ be the $D$-indicator of $\dd$, i.e., $q(t)$ is a formal power series
satisfying $\dd=q(D)$. Let $q^{-1}(t)$ be the compositional inverse of $q(t)$.
Then
\begin{eqnarray} \label{eq:EGF}
 \sum_{n \geq 0} \frac{p_n(x)}{n!} t^n = e^{x q^{-1}(t)}.
\end{eqnarray}

\subsection{Biorthogonal sequences} \

Generalized  \gonc polynomials are defined by a biorthogonality condition posted in the \gonc  interpolation
problem. Many  properties of these polynomials follow from a general theory of  \emph{sequences of polynomials biorthogonal
to a sequence of linear functionals}. The idea behind this theory is well-known (for examples, see \cite{BB58, Davis63}),
and an explicit description for the differential operator $D$  is given in \cite[section 2]{KuYan}.
Here we briefly describe this theory with a general delta operator $\dd$. The proofs are analogous to the ones in
\cite{KuYan} and hence omitted.

Let $\dd$ be a delta operator with the basic  sequence    $\PPc = (p_n(x))_{n \ge 0}$.
Let $\Phi_i$, $i=0,1,2,\ldots, $ be  a sequence of shift-invariant operators  of the form
$\sum_{j \ge 0} a_j^{(i)} \iter{\dd}{i+j}$, where $(a_j^{(i)}) \in \KK $
and  $a_0^{(i)} \ne 0$.  Then we have:
\begin{theorem} \label{thm:biortho}
\begin{enumerate}[label={\rm (\arabic{*})},leftmargin=0cm,itemindent=.5cm,labelwidth=\itemindent,labelsep=0cm,align=left]
\item\label{thm:biortho(1)} There exists a unique sequence $\FFc = (f_n(x))_{n \ge 0}$ of pol\-y\-no\-mi\-als  such that $f_n(x)$
is of degree $n$
and
$$
\varepsilon_0(\Phi_i(f_n(x))) = n! \delta_{i,n} \qquad  \text{ for all } i,n \in \NNb
$$
In addition, for every $n$ we have
$$
f_n (x)= \frac{n!}{a_0^{(0)} \cdots a_0^{(n)}} \fixed[0.15]{ \text{ }} {\det}_{\KK[x]}(\Lambda^{(n)}),
$$
with $\Lambda^{(n)}$ the $(n+1)$-by-$(n+1)$ matrix whose $(i,j)$-entry, for $0 \le i, j \le n$, is given by
$$
\lambda_{i,j}^{(n)} :=
\left\{
\begin{array}{ll}
\! a_{j-i}^{(i)},          & \text{if } 0 \leq  i \le \min(j,n-1) \\
\! \frac{1}{j!} \fixed[0.15]{ \text{ }} p_j(x), & \text{if } i = n \\
\! 0,                        & \text{otherwise.}
\end{array}
\right.\!
$$
\end{enumerate}
\begin{enumerate}[label={\rm (\arabic{*})},resume]
\item\label{thm:biortho(2)} The sequence $\FFc$ defined above forms a basis of $\KK[x]$.  For every $f(x) \in \KK[x]$ it holds
    \begin{equation}
    \label{equ:expansion_formula_for_biorthogonal_families}
    f(x) = \sum_{n \ge 0} \frac{\varepsilon_0(\Phi_i(f))}{n\fixed[0.15]{ \text{ }} !} \fixed[0.15]{ \text{ }} f_n(x)
      = \sum_{n = 0}^{\deg(f)} \frac{\varepsilon_0(\Phi_i(f))}{n\fixed[0.15]{ \text{ }} !} \fixed[0.15]{ \text{ }} f_n(x).
    \end{equation}
    \end{enumerate}
\end{theorem}

\subsection{Generalized \texorpdfstring{Gon\v{c}arov}{Goncarov} polynomials} \

Let $\ZZc = (z_i)_{i \ge 0}$ be a fixed sequence with values in $\KK$; in this context, we may refer to $\ZZc$ as
an (interpolation) $\KK$-grid (or simply  a grid),
and to the scalars $z_i \in \KK$ as (interpolation) nodes. For every $a \in \KK$, $E_a$ is an invertible shift-invariant
operator and  hence can be expressed as  $E_a=f_a(\dd)$ where $f_a(t) \in \KK \llb t \rrb$ with $f_a(0)\neq 0$.
It follows from Theorem \ref{thm:biortho} that there
is a unique sequence of polynomials  $\TTc = (\tt_n(x))_{n \ge 0} $  biorthogonal
to the sequence of operators $\{ \Phi_i=E_{z_i} \iter{\dd}{i}: i \ge 0 \}$. More precisely,
$t_n(x)$ satisfies
\begin{eqnarray} \label{equ:biortho}
 \epsilon_{z_i} \iter{\dd}{i}(t_n(x))  = n! \fixed[0.15]{ \text{ }} \delta_{i,n}.
\end{eqnarray}

\begin{definition}
We call the polynomial sequence $\TTc= (\tt_n(x))_{n \ge 0} $ determined by Eq.~\eqref{equ:biortho}
  the \emph{sequence of generalized \gonc polynomials}, or the \emph{generalized \gonc basis}, associated with the pair $(\dd,  \ZZc)$, and
  $t_n(x)$ the $n$-th generalized \gonc polynomial relative to the same pair.
  Accordingly, \eqref{equ:biortho} will be referred to as the biorthogonality property of gen\-er\-al\-ized \gonc bases.
\end{definition}

By Theorem \ref{thm:biortho}(2), for any   polynomial $f(x) \in \KK[x]$,   we have the expansion formula
\begin{eqnarray} \label{equ:expansion_formula}
f(x) = \sum_{i \ge 0} \frac{\epsilon_{z_i} \iter{\dd}{i}(f)} {i\fixed[0.15]{ \text{ }} !}  \tt_i (x)
= \sum_{i = 0}^{\deg(f)} \frac{\epsilon_{z_i} \iter{\dd}{i}(f)}{i\fixed[0.15]{ \text{ }}!} \fixed[0.15]{ \text{ }} \tt_i (x).
\end{eqnarray}
In particular, the solution of the generalized \gonc interpolation \eqref{eq:goncarov-interpolation} described in Section 1 is
given by the polynomial
\[
 p(x) = \sum_{i=0}^n \frac{ b_i}{i!} t_i(x).
\]
%
In some cases, to emphasize the dependence of $\TTc$ on $\dd$ and $\ZZc$, we  write $t_n(x)$  as  $\tt_n(x\fixed[0.25]{ \text{ }}; \dd, \ZZc)$.
 When the delta operator is the differentiation $D$, we get the classical \gonc polynomials, which were studied in \cite{KuYan}.
We reserve the symbols $g_n(x)$ and $g_n(x\fixed[0.25]{ \text{ }};\ZZc)$, respectively, for the classical \gonc polynomials
  to avoid confusion when we compare the results of generalized \gonc polynomials with those of the classical case.
Another special case that has been considered  is the
difference \gonc polynomials \cite{KSY}, for
which $\dd$ is the backward  difference operator $\Delta_{0,-1}=I-E_{-1}$.  The present paper is the first one
describing the theory of biorthogonal polynomials with an arbitrary delta operator, and hence connecting the
theory of interpolation to finite operator calculus. In the next sections we will describe the
algebraic properties of the generalized  \gonc polynomials, and reveal a deeper connection between binomial enumerations
and structure of order statistics.

We remark that every generalized \gonc basis is a sequence of biorthogonal polynomials, but the converse is not true in general.
%
%
\section{Algebraic properties of generalized \texorpdfstring{Gon\v{c}arov}{Goncarov} polynomials}
\label{sec:algebraic_properties_of_GGPs}

 Let $\ZZc = (z_i)_{i \ge 0}$ be a fixed $\KK$-grid. We
 denote by $\ZZc^{(j)}$ the grid whose $i$-th term is the $(i+j)$-th
term of $\ZZc$, and call $\ZZc^{(j)}$ the $j$-th shift of the grid $\ZZc$.
The zero grid, herein denoted by $\OOc$, is the one with $z_i=0$ for all $i$.
\begin{proposition}
\label{prop:trivialities_on_generalized_Goncarov_polynomials}
If $\TTc = (\tt_n(x))_{n \ge 0}$ is the generalized \gonc basis associated with the pair $(\dd, \ZZc)$,
then $\tt_0(x) = 1$ and $\tt_n(z_0) = 0$ for all $n \geq 1$.
\end{proposition}
\begin{proof}
This follows from the biorthogonality property \eqref{equ:biortho} with $i=0$.
\end{proof}
The next proposition is a generalization of the differential relations satisfied by the classical  \gonc polynomials \cite[p.~23]{KuYan}.
\begin{proposition}
\label{prop:differential_formula}
Let $\TTc = (t_n(x))_{n \ge 0}$ be the generalized \gonc basis associated with the pair $(\dd, \ZZc)$.
 Fix $j \in \NNb$ and define  for each $n \in \NNb$ the polynomial   $t_n^{(j)}(x)$ by letting
\begin{equation}
\label{equ:differentied_goncarov_bases}
t_n^{(j)}(x):= \frac{1}{(n+j)_{(j)}} \fixed[0.15]{ \text{ }} \iter{\dd}{j} \fixed[0.15]{ \text{ }} t_{n+j}(x),
\end{equation}
where $n_{(j)}=n(n-1)\cdots (n-j+1)$ is the $j$-th lower factorial function.
Then, $(t_n^{(j)}(x))_{n \ge 0}$ is the generalized \gonc basis associated with the pair $(\dd, \ZZc^{(j)})$. In particular, we have
\begin{eqnarray} \label{eq:differential}
 \iter{\dd}{j} t_n(x) = n_{(j)} t_{n-j}^{(j)}(x).
\end{eqnarray}

\end{proposition}
\begin{proof}
First, notice that $t_n^{(j)}(x)$ is a polynomial  of degree $n$, since
  $t_{n+j}(x)$ is a polynomial of degree $n+j$
and delta operators reduce degrees by one.

Next, pick $i,n \in \NNb$ with $i \le n$, and let $z_i^{(j)}$ denote the $i$-th node of the grid $\ZZc^{(j)}$.
We just need to verify that
$$
\epsilon_{z_i^{(j)}}( \iter{\dd}{i} (t_n^{(j)}(x))) = n! \fixed[0.15]{ \text{ }} \delta_{i,n}.
$$
Since $z_i^{(j)} = z_{i+j}$ and $\delta_{i,n} = \delta_{i+j,n+j}$, the above equation is equivalent to
$$
\frac{1}{(n+j)_{(j)}} \fixed[0.15]{ \text{ }} \epsilon_{z_{i+j}}(\iter{\dd}{i+j} (t_{n+j}(x)))
= n! \fixed[0.15]{ \text{ }} \delta_{i+j,n+j},
$$
which follows from the equations  $(n+j)_{(j)} n! = (n+j)!$ and
$\epsilon_{z_l}(\iter{\dd}{l} t_k(x)) = k! \fixed[0.15]{ \text{ }} \delta_{l,k}$ for all $l,k \in \NNb$.
The last statement is obtained by replacing $n+j$ by $n$ in \eqref{equ:differentied_goncarov_bases}.
\end{proof}
Following Proposition \ref{prop:differential_formula} we see that
sequences of polynomials of binomial type are a special
case of generalized \gonc polynomials.
%
\begin{proposition}
\label{prop:basic_sequence_only_if_generalized_goncarov}
The basic sequence of the delta operator $\dd$ is the generalized \gonc basis associated with the pair $(\dd, \OOc)$.
\end{proposition}
\begin{proof}
Let $(p_n(x))_{n \ge 0}$ be the basic sequence of the delta operator $\dd$.
Then iterating the equation $\dd(p_n) = n\fixed[0.15]{ \text{ }} p_{n-1}$ yields $\iter{\dd}{i} (p_n(x))= n_{(i)} p_{n-i}(x)$, which, when
evaluated at $x=0$, is $n!p_0(x)=n!$ if $n=i$, and $n_{(i)}p_{n-i}(0)=0$ if $i \neq n$.
\end{proof}
\begin{corollary}
\label{cor:characterization_of_binomial_sequences}
Let $\PPc= (p_n(x))_{n \ge 0}$ be a  sequence of polynomials with $\deg(p_i)=i$ for all $i$. Then, $\PPc$ is
of binomial type if and only if $\PPc$ is the generalized \gonc basis associated with the pair $(\dd, \OOc)$
for a suitable choice of $\dd$,.
\end{corollary}
\begin{proof}
The necessity follows from Proposition \ref{prop:basic_sequence_only_if_generalized_goncarov} and
Theorem 1(b) of  \cite{RoKaOd73}, which states that any sequence of polynomials of binomial type is a basic
sequence for some delta operator.

Conversely, let $\PPc = (p_n(x))_{n \ge 0}$ be  the generalized \gonc basis associated
with the pair $(\dd, \OOc)$. If $(p_n^{(1)}(x))_{n \ge 0}$ denotes the generalized \gonc basis associated with
the pair $(\dd, \OOc^{(1)})$, then by Proposition \ref{prop:differential_formula},
$\dd(p_n) = n \fixed[0.15]{ \text{ }} p_{n-1}^{(1)}$  for all
$n \geq 1$,  which in turn implies
$\dd(p_n) = n \fixed[0.15]{ \text{ }} p_{n-1}$, since $\OOc^{(1)} = \OOc$. This, together with
Proposition \ref{prop:trivialities_on_generalized_Goncarov_polynomials}, concludes the proof.
\end{proof}
%
%
%

Next we investigate the behavior of a generalized \gonc basis with respect to a transformation of
the interpolation grid. Proposition \ref{prop:invariance_under_translation} extends
 the shift-invariance  property  for classical  \gonc polynomials  and
difference \gonc polynomials.


\begin{proposition}
\label{prop:invariance_under_translation}
Let $\WWc = (w_i)_{i \ge 0}$ be a translation of the grid $\ZZc$ by $\xi \in \KK$,  i.e.,
 $w_i = z_i + \xi$ for all $i$.  Assume that $\TTc = (t_n(x))_{n \ge 0}$ and $\HHc = (h_n(x))_{n \ge 0}$ are
the generalized \gonc bases associated with the pairs $(\dd, \ZZc)$ and $(\dd, \WWc)$, respectively.
Then, $h_n(x+\xi) = t_n(x)$ for all $n$.
\end{proposition}
\begin{proof}
 Clearly  $h_n(x+\xi)=E_\xi(h_n(x))$ is a polynomial  of degree $n$, so by
the uniqueness of the \gonc basis associated with the pair $(\dd, \ZZc)$, it suffices to prove that
$$
\epsilon_{z_i}(\iter{\dd}{i} \fixed[0.25]{ \text{ }} E_\xi(h_n(x))) = n! \fixed[0.15]{ \text{ }} \delta_{i,n}.
$$
Note  that $\iter{\dd}{i} \fixed[0.25]{ \text{ }} E_\xi = E_\xi \fixed[0.15]{ \text{ }} \iter{\dd}{i}$
because any two shift-invariant operators commute.
Therefore
$$
\epsilon_{z_i}(\iter{\dd}{i}(E_\xi(h_n(x)))) = \epsilon_{z_i} \fixed[0.25]{ \text{ }} E_\xi(\iter{\dd}{i}(h_n(x))) =
\epsilon_{z_i+\xi}(\iter{\dd}{i}(h_n(x))) = \epsilon_{w_i}(\iter{\dd}{i}(h_n(x))) = n! \delta_{i,n}.
$$
\end{proof}

\begin{proposition}
\label{prop:shift_by_AP}
Fix $\xi \in \KK$ and let $\HHc = (h_n(x))_{n \ge 0}$ be the generalized \gonc basis associated with the pair
$(\dd, \WWc)$, where $\WWc=(w_i)_{i \geq 0} $ is the grid given by $w_i=z_i + i \xi$ for all $i \ge 0$.
Then, $\HHc$ is also the generalized \gonc
basis associated with the pair $(E_\xi \fixed[0.25]{ \text{ }} \dd, \ZZc)$.
\end{proposition}
\begin{proof}
One checks that
$$
\epsilon_{z_i}(\iter{(E_\xi \fixed[0.25]{ \text{ }} \dd)}{i}h_n(x)) = n! \fixed[0.15]{ \text{ }} \delta_{i,n}.
$$
To this end, first observe that $
\iter{(E_\xi \fixed[0.25]{ \text{ }} \dd)}{i} = \iter{E_{\xi}}{i} \fixed[0.25]{ \text{ }} \iter{\dd}{i}
= E_{i \xi} \fixed[0.15]{ \text{ }}  \iter{\dd}{i}$,
since any two shift-invariant operators commute
and $E_a \fixed[0.25]{ \text{ }} E_b = E_{a+b}$ for all $a,b \in \KK$. It follows that
$$
\epsilon_{z_i}(\iter{(E_\xi \fixed[0.25]{ \text{ }} \dd)}{i}h_n(x)) = \epsilon_{z_i} \fixed[0.15]{ \text{ }}
E_{i\xi}(\iter{\dd}{i}h_n(x)) = \epsilon_{z_i+i\xi}(\iter{\dd}{i}h_n(x)) =\epsilon_{w_i}(\iter{\dd}{i}h_n(x))
=n! \delta_{i,n}.
$$
The last equation is true  because $\HHc$ is the generalized \gonc basis associated with  $(\dd, \WWc)$.
\end{proof}

The next proposition gives a relation between the generalized \gonc polynomials and the basic polynomials of the
same delta operator. It provides  a linear recurrence that can be used to compute efficiently  the explicit
formulas for the generalized \gonc basis if the basic sequence is known,
for example, as in the classical case where
the basic sequence is $(x^n)_{n \ge 1}$, or in the case of difference \gonc polynomials where the $n$-term of the
basic sequence is the lower factorial $(x-n+1)_{(n)}$, see \cite{KSY}.

\begin{proposition}
\label{prop:recursive_relation}
Let $\TTc = (t_n(x))_{n \ge 0}$ be the generalized \gonc basis associated with the pair $(\dd, \ZZc)$, and
let $(p_n(x))_{n \ge 0}$ be the  sequence of basic polynomials of the delta operator $\dd$. Then, for
all $n \in \NNb$ it holds
\begin{equation*}
p_n(x) = \sum_{i=0}^n \binom{n}{i} p_{n-i}(z_i) \fixed[0.25]{ \text{ }} t_i(x),
\end{equation*}
and hence
\begin{equation}
\label{equ:linear_recursion}
t_n(x) = p_n(x) - \sum_{i=0}^{n-1} \binom{n}{i} p_{n-i}(z_i) \fixed[0.25]{ \text{ }} t_i(x).
\end{equation}
\end{proposition}
\begin{proof}
Let $n \in \NNb$.
Substituting $f(x)$ with $p_n(x)$ in Eq. \eqref{equ:expansion_formula} we obtain
\begin{equation*}
p_n(x) = \sum_{i=0}^n \frac{\varepsilon_{z_i}(\iter{\dd}{i}p_n)}{i \fixed[0.15]{ \text{ }} !} \fixed[0.15]{ \text{ }} t_i(x)
= \sum_{i=0}^n \frac{(n)_{(i)} \fixed[0.15]{ \text{ }} p_{n-i}(z_i)}{i \fixed[0.15]{ \text{ }} !} \fixed[0.15]{ \text{ }}
\fixed[0.15]{ \text{ }} t_i(x) = \sum_{i=0}^n \binom{n}{i} p_{n-i}(z_i)   \fixed[0.25]{ \text{ }} t_i(x),
\end{equation*}
where we use the fact that
$
\iter{\dd}{i}(p_n(x)) = (n)_{(i)} \fixed[0.15]{ \text{ }} p_{n-i}(x)$ for all $i=0, 1, \ldots, n$.

\end{proof}

Proposition \ref{prop:binomial_expansion_of_GGPs} generalizes the binomial expansion for classical \gonc polynomials.

\begin{proposition}
\label{prop:binomial_expansion_of_GGPs}
Let $(t_n^{(j)}(x))_{n \ge 0}$ be the generalized \gonc basis associated with the pair $(\dd, \ZZc^{(j)})$,
and let $(p_n(x))_{n \ge 0}$ be the  sequence of basic polynomials of the delta operator $\dd$.
Then, for all $\xi \in \KK$ and $n \in \NNb$ we have the following ``binomial identity'':
\begin{equation}
\label{equ:binomial_expansion_formula}
t_n(x+\xi)= t_n^{(0)}(x+\xi) = \sum_{i=0}^n \binom{n}{i} t_{n-i}^{(i)}(\xi) \fixed[0.25]{ \text{ }} p_i(x).
\end{equation}
In particular, letting $\xi=0$ we have
\begin{equation}\label{equ:zero}
t_n(x)=t_n^{(0)}(x) = \sum_{i=0}^n \binom{n}{i} t_{n-i}^{(i)}(0) \fixed[0.25]{ \text{ }} p_i(x).
\end{equation}
\end{proposition}
\begin{proof}
Fix $\xi \in \KK$ and $n \in \NNb$. Since $(p_i(x))_{0 \le i \le n}$ is a basis of the linear subspace of
$\KK[x]$ of polynomials of degree $\le n$, there exist $c_0, \ldots, c_n \in K$ such that
$E_\xi t_n(x) = \sum_{i=0}^n c_i \fixed[0.25]{ \text{ }} p_i(x)$, where $t_n := t_n^{(0)}$.

Pick  an integer $j \in [ 0, n ]$.  One computes
\begin{equation}
\label{equ:binomial_expansion_proof}
\epsilon_\xi(\iter{\dd}{j}t_n(x))
    = \epsilon_0(E_\xi \fixed[0.25]{ \text{ }} \iter{\dd}{j}(t_n(x))) = \epsilon_0(\iter{\dd}{j}(E_\xi t_n(x))) =
    \sum_{i=0}^n c_i \fixed[0.25]{ \text{ }} \epsilon_0(\iter{\dd}{j} \fixed[0.25]{ \text{ }} p_i(x)).
\end{equation}
From Proposition \ref{prop:differential_formula} we have $\iter{\dd}{j} t_n(x) = (n)_{(j)} \fixed[0.25]{ \text{ }} t_{n-j}^{(j)}(x)$,
so that $\epsilon_\xi(\iter{\dd}{j}t_n(x)) = (n)_{(j)} \fixed[0.25]{ \text{ }} t_{n-j}^{(j)}(\xi)$.
On the other hand, the sequence $(p_n(x))_{n \ge 0}$, being the basic polynomials of $\dd$, satisfies
$\epsilon_0(\iter{\dd}{j} \fixed[0.25]{ \text{ }} p_i) = i\fixed[0.15]{ \text{ }}! \fixed[0.25]{ \text{ }} \delta_{i,j}$
for all $i \in \NNb$.
Combining the above we obtain from \eqref{equ:binomial_expansion_proof} that
$$
c_j = \frac{(n)_{(j)}}{j\fixed[0.15]{ \text{ }}!} \fixed[0.25]{ \text{ }} t_{n-j}^{(j)}(\xi) = \binom{n}{j} t_{n-j}^{(j)}(\xi),
$$
which proves \eqref{equ:binomial_expansion_formula} .
\end{proof}

\begin{corollary}
\label{cor:binomial_expansion_of_GGPs_1}
Assume $\ZZc$ is a constant grid, namely $z_i = z_j$ for all $i,j$. Let $(t_n(x))_{n \ge 0}$ be the generalized \gonc basis
associated with the pair $(\dd, \ZZc)$, and $(p_n(x))_{n \ge 0}$ the  sequence of basic polynomials of the delta operator $\dd$. Then,
for all $\xi \in K$ and $n \in \NNb$ we have
\begin{equation}
\label{equ:binomial_expansion_formula_cor1}
t_n(x+\xi) = \sum_{i=0}^n \binom{n}{i} t_{n-i}(\xi) \fixed[0.25]{ \text{ }} p_i(x)\quad\textrm{and}\quad
t_n(x) = \sum_{i=0}^n \binom{n}{i} t_{n-i}(0) \fixed[0.25]{ \text{ }} p_i(x).
\end{equation}
\end{corollary}
\begin{proof}
Immediate by Proposition \ref{prop:binomial_expansion_of_GGPs} and the fact that $\ZZc = \ZZc^{(j)}$ for all $j$.
\end{proof}

The next proposition gives an extension of the integral formula for classical \gonc polynomials, see \cite[p. 23]{KuYan}.
\begin{proposition}\label{prop:integral}
Let $\dd$ be a delta operator. Then $\dd$ has a right inverse, i.e., there exists a linear operator $\dd^{-1}: \KK[x] \to \KK[x]$ such that
$\dd(\dd^{-1}(f(x))) = f$ for all $f (x)\in \KK[x]$, $\deg(\dd^{-1}(f(x))) = 1 + \deg(f(x))$ for $f \ne 0$, and $\dd^{-1}(f)(0) = 0$.
\end{proposition}
\begin{proof}
Let $\PPc = (p_n(x))_{n \ge 0}$ denote the  sequence of basic polynomials of the delta operator $\dd$. Then we can define an
operator $\dd^{-1}: \KK[x] \to \KK[x]$ as follows: Given a polynomial $f(x) \in \KK[x]$ of degree $n$, let  $a_0, \ldots, a_n \in K$
be such that $f(x) = \sum_{i=0}^n a_i p_i(x)$, set
$$
\dd^{-1}(f(x)) := \sum_{i=0}^n \frac{a_i}{i+1} p_{i+1}(x).
$$
It is seen that $\dd^{-1}$ is a linear operator on $\KK[x]$, and since $\dd(p_n(x)) = n\fixed[0.15]{ \text{ }} p_{n-1}(x)$
and $p_n(0) = 0$ for all $n \in \NNb^+$, we have as well that $\dd(\dd^{-1}(f(x))) = f(x)$ and $\dd^{-1}(f)(0) = 0$ for
every $f(x) \in \KK[x]$. The rest is trivial, when considering that $p_n(x)$ is, for each $n \in \NNb^+$, a polynomial of degree $n$.
\end{proof}
Thus, we have the following generalization of the integral formula for classical \gonc polynomials, see \cite[p. 23]{KuYan}.
\begin{proposition}
Let $\dd$ be a delta operator and $\dd^{-1}$ its right inverse (which exists by Proposition \ref{prop:integral}),
and let $(t_n(x))_{n \ge 0}$ be the generalized \gonc basis associated with the pair $(\dd, \ZZc)$. Then, for every
$n, k \in \NNb$ with $k \le n$, it holds that
$t_n(x) = (n)_{(k)} \cdot \mathcal{I}_k(t_{n-k}^{(k)}(x))$, where $(t_i^{(n)}(x))_{n \ge 0}$ is the generalized \gonc
basis associated with the pair $(\dd, \ZZc^{(k)})$ and $\mathcal{I}_k$ the linear operator $\prod_{i=0}^{k-1} (1 - \varepsilon_{z_i}) \dd^{-1}$.
\end{proposition}
\begin{proof}
Fix $n, k \in \NNb$ with $0 \le k \le n$. If $k = 0$, the claim is trivial, because $\mathcal{I}_0$ is the identity operator.
Now, suppose we have already confirmed that the statement is true for $0 \le k < n$; then, by induction, we are just left to
show that it continues to be true for $k+1$.

For this, it is enough to prove that $t_{n-k}^{(k)}(x) = (n-k) \cdot (1 - \varepsilon_{z_k}) \dd^{-1} (t_{n-k-1}^{(k+1)}(x))$.
It follows from the facts that (i) both $t_{n-k}^{(k)}(x)$ and $(1 - \varepsilon_{z_k}) \dd^{-1} (t_{n-k-1}^{(k+1)}(x))$ are zero
when evaluated at $z_k$, and (ii) we have $\dd(t_{n-k}^{(k)}(x)) = (n-k) t_{n-k-1}^{(k+1)}(x)$ by
Proposition \ref{prop:differential_formula}, and
$$
\dd((1 - \varepsilon_{z_k}) \dd^{-1} (t_{n-k-1}^{(k+1)}(x))) = \dd(\dd^{-1} t_{n-k-1}^{(k+1)}(x)) - \dd(\varepsilon_{z_k}
\dd^{-1} (t_{n-k-1}^{(k+1)}(x))) = t_{n-k-1}^{(k+1)}(x),
$$
where we used Proposition \ref{prop:integral} and the fact that $\varepsilon_{z_k} \dd^{-1} (t_{n-k-1}^{(k+1)}(x))$ is
a constant (and $\dd$ applied to a constant is zero).
\end{proof}
%
%
In addition, we have the following generalization of the ``perturbation formulas'' obtained in \cite[p. 24]{KuYan} and \cite[p. 5]{KSY}.
\begin{proposition}
Let $\dd$ be a delta operator, and let $\ZZc = (z_i)_{i \ge 0}$ and $\ZZc^\prime = (z_i^\prime)_{i \ge 0}$ be $\KK$-grids such that
$z_k \ne z_k^\prime$ for a given $k \in \NNb$ and $z_i = z_i^\prime$ for $i \ne k$. Then we have, for $n > k$, that
\begin{equation}\label{equ:perturbated_polys}
\tt_n(x\fixed[0.25]{ \text{ }}; \dd, \ZZc^\prime) =
\tt_n(x\fixed[0.25]{ \text{ }}; \dd, \ZZc) - \binom{n}{k} \tt_{n-k}(z_k^\prime\fixed[0.25]{ \text{ }}; \dd, \ZZc^{(k)}) \fixed[0.25]{ \text{ }} t_k(x\fixed[0.25]{ \text{ }}; \dd, \ZZc),
\end{equation}
while $\tt_n(x\fixed[0.25]{ \text{ }}; \dd, \ZZc^\prime) = \tt_n(x\fixed[0.25]{ \text{ }}; \dd, \ZZc)$ for $n \le k$.
\end{proposition}
\begin{proof}
Let $n \in \NNb$ and denote by $f_n(x)$ the polynomial on the right-hand side of \eqref{equ:perturbated_polys}. The claim is
straightforward if $n \le k$, essentially because we get by Theorem \ref{thm:biortho}\ref{thm:biortho(1)}
that $\tt_n(x\fixed[0.25]{ \text{ }}; \dd, \ZZc)$ and $\tt_n(x\fixed[0.25]{ \text{ }}; \dd, \ZZc^\prime)$ depend
only on the first $n$ nodes of $\ZZc$ and $\ZZc^\prime$, respectively. Accordingly, assume in what follows that $n > k$ and fix $i \in \NNb$.
By the unicity of the generalized \gonc basis associated to the pair $(\dd, \ZZc^\prime)$, we just have to prove
that $\varepsilon_{z_i^\prime}(\iter{\dd}{i}f_n(x)) = n! \fixed[0.2]{ \text{ }}\delta_{i,n}$. This is immediate if $i > n$,
since then $\iter{\dd}{i}f_n(x) = 0$, and for $i \le n$ it is a consequence of Proposition \ref{prop:differential_formula}
(we omit further details).
\end{proof}
We conclude the present section by proving that generalized \gonc bases obey an Appell relation, which extends an analogous
result from \cite[Section 3]{KuYan}.
%
\begin{proposition}
\label{prop:appell_relation}
Let $(t_n(x))_{n \ge 0}$ be the generalized \gonc basis associated with the pair $(\dd, \ZZc)$. In addition, denote by $d(t)$
the compositional inverse of the $D$-indicator of $\dd$ in $\KK \llb t\rrb $.
Then the following identity holds,
\begin{equation}
\label{equ:appell_relation}
e^{x \fixed[0.15]{ \text{ }} d(t)} = \sum_{n \ge 0} \frac{1}{n!} t_n(x)\fixed[0.15]{ \text{ }} e^{z_n d(t)} t^n.
\end{equation}
In particular, if $\dd = D$ then $
e^{xt} = \sum_{n \ge 0} \frac{1}{n!} t_n(x) \fixed[0.15]{ \text{ }} e^{z_n t} t^n$.
\end{proposition}
%
%
\begin{proof}
Let $(p_n(x))_{n \ge 0}$ the sequence of basic polynomials  of $\dd$. By Proposition \ref{prop:recursive_relation}
$$
p_n(x) = \sum_{i=0}^n \binom{n}{i} p_{n-i}(z_i)\fixed[0.15]{ \text{ }} t_i(x)  \qquad \text{for all } n \in \NNb,
$$
whence we get that
\begin{equation}
\label{equ:expanding_towards_Appell}
\sum_{n \ge 0} \frac{p_n(x) }{n!} t^n
= \sum_{n \ge 0} \sum_{i=0}^n \frac{1}{i\fixed[0.15]{ \text{ }}! (n-i)!} p_{n-i}(z_i) \fixed[0.15]{ \text{ }} t_i(x)
\fixed[0.15]{ \text{ }} t^n
= \sum_{i \ge 0} \left(\frac{1}{i\fixed[0.15]{ \text{ }}!} t_i(x)\fixed[0.15]{ \text{ }} t^i \sum_{j \ge 0}
\frac{p_j(z_i)}{j\fixed[0.1]{ \text{ }}!} t^j \right)\!.
\end{equation}
On the other hand,   Formula \eqref{eq:EGF} gives
$$
e^{x \fixed[0.15]{ \text{ }} d(t)} = \sum_{n \ge 0} \frac{p_n(x)}{n!}\fixed[0.15]  t^n,
$$
which, together with \eqref{equ:expanding_towards_Appell}, implies \eqref{equ:appell_relation}.
The rest is trivial, when considering that the $D$-indicator of $D$ over $\KK\llb t \rrb$ is just $t$.
\end{proof}
Using Proposition \ref{prop:appell_relation}, we obtain   the following characterization of
sequences of binomial type, which is complementary to Corollary \ref{cor:characterization_of_binomial_sequences}.
\begin{proposition}
Let $\TTc = (t_n(x))_{n \ge 0}$ be the generalized \gonc basis associated with the pair $(\dd, \ZZc)$.
Then $\TTc$ is  of binomial type if and only if $\ZZc$ is an arithmetic grid of initial term $0$.
\end{proposition}
\begin{proof}
Suppose first that $\TTc$ is a sequence of binomial type. We get by \cite[Section 3, Corollary 3]{RoKaOd73}
that there is a formal power series $f \in \KK\llb t \rrb$ with $f(0) = 0$ and $f^\prime(0) \ne 0$ such that
\begin{equation}
\label{equ:generating_function_for_sequences_of_binomial_type}
e^{x f(t)} = \sum_{n=0}^\infty \frac{t_n(x)}{n!}  t^n.
\end{equation}
On the other hand, we have from Proposition \ref{prop:appell_relation} that
\begin{equation}
\label{equ:exponential_generating_function_for_GGP}
e^{x \fixed[0.15]{ \text{ }} d(t)} = \sum_{n \ge 0} \frac{1}{n!} t_n(x)\fixed[0.15]{ \text{ }} e^{z_n d(t)} t^n,
\end{equation}
where $d(t)$ is the compositional inverse of the $D$-indicator of the delta operator $\dd$ in $\KK\llb t \rrb$.

Let $h \in \KK \llb t \rrb$ be the compositional inverse of $f$, which exists by the assumption that
$f(0) = 0$ and $f^\prime(0) \ne 0$ \cite{Stan2}.  Then $h(f(t)) = f(h(t)) = t$.
Using the change of variable $y \mapsto h(d(t))$ in \eqref{equ:generating_function_for_sequences_of_binomial_type} yields that
$$
e^{x d(t)} = \sum_{n=0}^\infty \frac{1}{n!} t_n(x)\fixed[0.15]{ \text{ }}(h(d(t)))^n.
$$
Comparing this with \eqref{equ:exponential_generating_function_for_GGP} implies, for all $n \in \NNb$, that
$$
(h(d(y)))^n = e^{z_n d(y)} y^n,
$$
which holds as an identity between formal power series in $\KK \llb t \rrb$, and is in turn possible, by the
further change of variables $t \mapsto d^{-1}(t)$, if and only if
\begin{equation}
\label{equ:comparing_coefficients}
(h(t))^n = e^{z_n t} (d^{-1}(t))^n.
\end{equation}
Combining with $(h(t))^{n+1} = e^{z_{n+1} t} (d^{-1}(t))^{n+1}$ for all $n \in \NNb$,  we find that
$$
h(t) = \frac{(h(t))^{n+1}}{(h(t))^n} = e^{(z_{n+1} - z_n)t} d^{-1}(t).
$$
It follows that $z_{n+1} - z_n$ is a constant independent of $n$, viz. there exists $b \in K$ such that
$z_{n+1} - z_n = b$ for all $n \in \NNb$. Then  $z_n = z_0 + nb$ for all $n \in \NNb$. But
evaluating \eqref{equ:comparing_coefficients} at $n = 0$ gives $1 = e^{z_0 t}$, which implies
$z_0 = 0$. Thus $\ZZc$ is an arithmetic grid of initial term $0$.

As for the converse, assume now that $\ZZc$ is an arithmetic grid of common difference $b \in \KK$ and initial
term $0$. Then  $\TTc$ is, by Proposition \ref{prop:shift_by_AP}, the generalized
\gonc basis associated with the pair $(E_b \fixed[0.15]{ \text{ }} \dd, \OOc)$, and hence by Corollary
\ref{cor:characterization_of_binomial_sequences}  is a sequence of binomial type.
\end{proof}

\section{A combinatorial formula for generalized \texorpdfstring{Gon\v{c}arov}{Goncarov} polynomials}
\label{sec:coefficients_of_GGPs}
Let $\dd$ be a delta operator and $\ZZc$ a $\KK$-grid. Assume  $\TTc = (t_n(x))_{n \ge 0}$ is the generalized
\gonc basis associated with the pair $(\dd, \ZZc)$. The main purpose of this  section is to provide a combinatorial
interpretation of the coefficients  of  $t_n(x)$.    By Eq.~\eqref{equ:zero} it is sufficient to consider only
the constant terms.  We will give an explicit
combinatorial formula of $t_n(0)$ as a summation of ordered partitions.

 Given a finite set $S$ with $n$ elements, an \textit{ordered partition},
or \textit{preferential arrangement}, of $S$ is an ordered list $(B_1, \ldots, B_k)$ of disjoint nonempty
subsets of $S$ such that $B_1\cup \cdots \cup B_k=S$.

If $\rho = (B_1,\ldots,B_k)$ is an ordered
partition of $S$, then we set $|\rho |= k$.  For every $i = 0, 1, \ldots, k$ we let $b_i := b_i(\rho) := |B_i|$
and $s_i := s_i(\rho) := \sum_{j=1}^i b_i$.  In particular, $s_0(\rho) = 0$.

Let $\mathcal{R}[n]$ be the set of all ordered partitions of the set $[n]:=\{ 1, 2, \ldots, n\}$.
It is shown in \cite[Theorem 4.2]{KuYan} that the constant coefficient for $g_n(x; \ZZc) $, the classical \gonc polynomial associated to $(D, \ZZc)$,
can be expressed as
\begin{eqnarray} \label{classical-GP}
 g_n(0; z_0, \ldots, z_{n-1}) = \sum_{\rho} (-1)^{|\rho|} \prod_{i=0}^{k-1} z_{s_i}^{b_{i+1}}
 = \sum_{\rho \in \mathcal{R}[n]} (-1)^{|\rho|} z_0^{b_1} \cdots z_{s_{k-1}}^{b_k}.
\end{eqnarray}
A similar formula holds for the  generalized \gonc polynomials associated to the pair $(\dd, \ZZc)$.

\begin{theorem} \label{thm:constant}
 Let $(t_n(x))_{n \geq 0}$ be the generalized \gonc basis associated with the pair $(\dd,\ZZc)$,
 and $(p_n(x))_{n\geq 0}$ be the sequence of basic polynomials
 of $\dd$. Then for $n \geq 1$,
 \begin{eqnarray} \label{constant-coefficient}
  t_n(0) =  \sum_{\rho \in \mathcal R[n]} (-1)^{|\rho|} \prod_{i=0}^{k-1} p_{b_{i+1}}(z_{s_i})
    = \sum_{\rho \in \mathcal R[n]} (-1)^{|\rho|} p_{b_1}(z_0) \cdots p_{b_k}(z_{s_{k-1}}).
 \end{eqnarray}
\end{theorem}
\begin{proof}
 Using Proposition \ref{prop:recursive_relation}  and noting that $p_n(0)=0$ for $n \geq 1$, we have
 \begin{eqnarray} \label{equ:constant-recurrence}
  t_n(0) = -\sum_{i=0}^{n-1} \binom{n}{i} p_{n-i}(z_i) t_i(0).
\end{eqnarray}
Denote by $\mathcal{T}(n)$, for $n \geq 1$, the  right-hand side of Eq.~\eqref{constant-coefficient}, and let $\mathcal{T}(0)=1$, which agrees with $t_0(0)$. We show by induction that $(\mathcal{T} (n))_{n\geq 0}$ satisfies the same
recurrence as \eqref{equ:constant-recurrence}, i.e.,
\begin{eqnarray} \label{sum-of-OP}
  \mathcal{T}(n) = -\sum_{i=0}^{n-1} \binom{n}{i} p_{n-i}(z_i) \mathcal{T}(i)
\end{eqnarray}
To see this, we divide $\mathcal{R}[n]$ into disjoint subsets $\mathcal{R}[n,i]$, where
$$
\mathcal{R}[n,i] := \{(B_1, \ldots,  B_k)  \in \mathcal{R}[n]: |B_k|=n-i\} \qquad
\text{ for }  i=0, 1, \ldots, n-1.
$$
Given $\rho \in \mathcal{R}[n,i]$ with a fixed last block $B_k$, we can write $\rho$ as the concatenation of $B_k$ and an ordered partition of a set with $i$ elements. So we get from the inductive hypothesis that
$$
 \sum_{\rho^\prime \in \mathcal{R}[i]} (-1)^{|\rho^\prime|} \prod_{i=0}^{k-2} p_{b_{i+1}}(z_{s_i}) = \mathcal{T}(i).
$$
Since there are $\binom{n}{i}$ ways to choose the elements of $B_k$, it follows that
the total contribution of ordered partitions in $\mathcal{R}[n,i]$ to $\mathcal{T}(n)$ is
$$
 -\binom{n}{i} p_{n-i}(z_i) \mathcal{T}(i).
$$
Then, summing over $i=0, 1, \ldots, n-1$ proves the desired recurrence \eqref{sum-of-OP}.
\end{proof}
One way of obtaining $t_n(x\fixed[0.15]{ \text{ }}; \dd, \ZZc)$ is to use the shift-invariance  property (Proposition
\ref{prop:invariance_under_translation}) as to write $t_n(x\fixed[0.15]{ \text{ }}; \dd, \ZZc) = t_n(0\fixed[0.15]{ \text{ }}; \dd, \ZZc-x)$. Another
way is to use Eq.~\eqref{equ:zero}:
$$
 t_n(x\fixed[0.15]{ \text{ }}; \dd, \ZZc) = \sum_{i=0}^n \binom{n}{i} t_{n-i}(0; \dd, \ZZc^{(i)}) p_i(x),
$$
and apply Theorem \ref{thm:constant} to each $t_{n-i}(0; \dd, \ZZc^{(i)})$.
Comparing this with the analogous equation for classical \gonc polynomials:
$$
 g_n(x\fixed[0.15]{ \text{ }}; \ZZc) = \sum_{i=0}^n \binom{n}{i}  g_{n-i}(0\fixed[0.15]{ \text{ }}; z_i, \ldots, z_{n-1})\fixed[0.15]{ \text{ }} x^i,
$$
we notice that $t_n(x\fixed[0.15]{ \text{ }}; \dd, \ZZc)$ can be obtained from $g_n(x; \ZZc)$ by replacing $x^i$
with $p_i(x)$  and $z_k^i$ by $p_i(z_k)$. For example, we have the following formulas, which the reader may want to compare with the ones for $g_n(x\fixed[0.15]{ \text{ }}; \ZZc)$ in \cite[p. 23]{KuYan}:
\begin{equation*}
\begin{split}
 t_0(x\fixed[0.15]{ \text{ }}; \dd, \ZZc) &=  1 ,\\
 t_1(x\fixed[0.15]{ \text{ }}; \dd, \ZZc) & = p_1(x)-p_1(z_0), \\
 t_2(x\fixed[0.15]{ \text{ }};\dd, \ZZc) & =  p_2(x)-2 p_1(z_1)p_1(x)+ 2 p_1(z_0)p_1(z_1)- p_2(z_0), \\
 t_3(x\fixed[0.15]{ \text{ }}; \dd, \ZZc) & =  p_3(x)-3p_1(z_2) p_2(x) + (6 p_1(z_1)p_1(z_2)-3p_2(z_1)) p_1(x)  \\
     &  - p_3(z_0) +3 p_2(z_0)p_1(z_2)-6 p_1(z_0) p_1(z_1)p_1(z_2) +3 p_1(z_0) p_2(z_1).
\end{split}
\end{equation*}
For a generic delta operator $\dd$ and an arbitrary grid $\ZZc$, the generalized \gonc polynomials  do not
usually have a simple closed formula. However, an interesting exception to this ``rule'' occurs when $\ZZc$ is an arithmetic progression with $z_i=a+bi$, in which case we refer to the corresponding generalized \gonc polynomials as $\dd$-Abel polynomials.
In fact, $\dd$-Abel polynomials can be expressed
in terms of the basic polynomials of $\dd$, as implied by the following:
\begin{theorem}\label{thm:delta_Abel}
 Let $\dd$ be a delta operator with basic sequence $(p_n(x))_{n \geq 0}$, and let $\ZZc$ be the
arithmetic grid $(a+bi)_{i \ge 0}$, where $a, b \in \KK$. Then
 the $\dd$-Abel polynomial $t_n(x;\dd, \ZZc)$ can be obtained by
 \begin{eqnarray} \label{delta-abel}
  t_n(x\fixed[0.15]{ \text{ }}; \dd, \ZZc) = \frac{(x-a)\fixed[0.2]{ \text{ }}  p_n(x-a-nb)} {x-a-nb}.
 \end{eqnarray}
\end{theorem}
\begin{proof}
 By the shift-invariance formula we have  $t_n(x; \dd, \ZZc) = t_n(x-a; \dd,  (bi)_{i\geq 0})$. Using
 Proposition \ref{prop:shift_by_AP}, the generalized \gonc polynomials associated to
 $(\dd, \WWc)$ with $w_i = bi$ is also the generalized \gonc polynomials associated to $(E_b \dd, \OOc)$.
 Let $q_n(x)$ be the basic
 sequence of $E_b \dd$. Hence we have $t_n(x-a; \dd, (bi)_{i \ge 0} )
 = q_n(x-a)$.

Now, we have from \cite[Theorem 4(3)]{MuRo70}, along with the fact that shift-invariant operators commute with each other, that
\begin{equation}\label{equ:mullin-rota_th4(3)}
q_n(x) = x \fixed[0.15]{ \text{ }} (E_b \dd)^{-n}(x^n) = x \fixed[0.15]{ \text{ }} E_{-nb} \dd^{-n}(x^n).
\end{equation}
On the other hand, a further application of \cite[Theorem 4(3)]{MuRo70} yields that $p_n(x) = x \dd^{-n}(x^n)$, and hence $\dd^{-1}(x^n) = p_n(x)/x$. Together with \eqref{equ:mullin-rota_th4(3)}, this in turn implies that
$$
q_n(x) = x E_{-nb} \!\left(\frac{p_n(x)}{x}\right) \!= \frac{x \fixed[0.2]{ \text{ }} p_n(x-nb)}{x-nb}.
$$
So putting it all together, Eq.~\eqref{delta-abel} follows immediately.
\end{proof}
We note that Niederhausen has also obtained formula \eqref{delta-abel} in \cite{Nied}, but with other means; he calls the procedure $\dd \mapsto E_a \dd$, for a fixed $a \in \KK$, the abelization of the delta operator $\dd$.

Bivariate extensions of $\dd$-Abel polynomials, which are solutions to a multivariate
 \gonc Interpolation problem with respect to an affine grid, are further studied and characterized in \cite{LY15}
 for $\dd=D$, and in \cite{LSY} for general delta operators.


\section{Reluctant functions and order statistics}

In the classical paper \emph{Finite Operator Calculus}, Rota, Kahaner, and
Odlyzko presented a unified theory of special polynomials via operator methods.
One open question arose from this algebraic theory is to find the ``statistical, probabilistic and
combinatorial interpretations of the identities" (of the polynomials),
see  Problem 5 of Section 14, \cite{RoKaOd73}.
For polynomial sequences of binomial type, Mullin and Rota  \cite{MuRo70} provided
a combinatorial interpretation through counting binomial type structures such as
\emph{reluctant functions}.
The  ideas of binomial enumeration and reluctant functions
also provide a  combinatorial setting for generalized \gonc polynomials: We
show in this section that generalized \gonc polynomials are the natural polynomial basis for counting
the number of  binomial type structures subject to a linear constraint on their order statistics.

To start, let $S$ and $X$ be finite disjoint sets, and $f: S \to S \cup X$ a function. We say that $f$
is a \textit{reluctant function} from $S$ to $X$ if, for every $s \in S$, there is a  positive integer
$k=k(s)$ such that $\iter{f}{k}(s) \in X$, in which case we refer to $\iter{f}{k}(s)$ as the \textit{final image}
of $s$ (under $f$).
It is easy to see that for any given $s$, the integer  $k(s)$, if it exists, is unique.

Accordingly, we take the \textit{final range} of $f$, here denoted by $\uimag(f)$, to be the set of all $\xi \in X$
such that $\xi$ is the final image of some $s \in S$. Given $\xi \in \uimag(f)$, we let the \textit{final inverse image}
of $\xi$, which we write as $f^{(-1)}(\xi)$, be the set of all the elements in $S$ whose final image is $\xi$.

From a combinatorial point of view, the final inverse image of an element $\xi \in X$ can be regarded in a
canonical way as a rooted forest: The nodes are just the elements of $S$, the roots are the elements of the
inverse image $f^{-1}(\xi)=\{s \in S: \ f(s) =\xi\}$.
In a rooted tree we say that a vertex is of depth $k$ if the unique path from $u$ to the root contains
$k$ edges. The root itself is of depth $0$.
Then for each $s_0 \in f^{-1}(\xi)$ and $k \in \NNb$ the vertices of depth $k$ in a tree  of $f^{(-1)}(\xi)$ rooted at $s_0$
are those elements $s \in S$ such that $\iter{f}{k}(s) = s_0$ and
hence $\iter{f}{k+1}(s)=\xi$.

The \textit{final coimage} of $f$ is the partition $\{f^{(-1)}(\xi): \xi \in \uimag(f)\}$ of $S$.
 Based on the above discussion the final coimage  carries over a natural structure, $T_f$, of a rooted forest
 defined on each block of the partition. Furthermore, each block of the final coimage can be further partitioned
 into  connected components (relative to $T_f$); the resulting partition is a refinement of the final coimage
 and has the additional property that each block has the structure of a rooted tree. This finer partition
 together with the rooted tree structure  is called the \textit{final preimage} of the reluctant function $f$.

\begin{remark}
What we call ``final range'', ``final coimage'', and ``final preimage'' of a reluctant function are actually
called ``range'', ``coimage'', and ``preimage'' by Mullin and Rota in \cite{MuRo70}. However, these latter
terms are already used in the everyday practice of mathematics with a different meaning.  We add the word ``final''
to avoid potential
misunderstanding.
\end{remark}
A \emph{binomial class} $\mathcal{B}$ of reluctant functions is defined as follows. To  every pair of
finite sets $S$ and  $X$ we assign a set
$F(S, X)$ of reluctant functions from $S$ to $X$, where $F(S, X)$ is isomorphic to $F(S', X')$ whenever $S$
is isomorphic to $S$ and
$X$ is isomorphic to $X'$.  Consequently, the size of $F(S, X)$ depends only on the
sizes of $S$ and $X$, but not the content of these sets.  Let $X \oplus Y$ stand for the disjoint union of $X$
and $Y$. For every reluctant function $f$ from $S$ to $X \oplus Y$, let $A=\{s \in S: \text{the final image of } s \text{
is in }X\}$ and $f_A$ is the
restriction of $f$ to $A$.  Similarly $f_{S \setminus A}$ is the restriction of $f$ to the set $S\setminus A$.
 The class $\mathcal{B}$  is a \emph{binomial class} if $f_A \in F(A, X)$, $f_{S\setminus A} \in F(S\setminus A, Y)$ and
 the above decomposition
 leads to a  natural isomorphism
\begin{equation} \label{BT-set}
\mu:  F(S, X \oplus Y) \rightarrow \bigcupdot_{A \subseteq S} (F(A, X) \otimes F(S \setminus A, Y))
\end{equation}
by letting $\mu(f) := (f_A,  f_{S\setminus A} )$, where $\otimes$ denotes the operation of piecing functions $f_A$ and $f_{S\setminus A}$ together,
and $\cupdot$ a disjoint union.

Let $p_n(x)$  denote the size of the set $F(S, X)$ when $|S|=n$ and $|X|=x$. Then for a binomial class,
$p_n(x)$ is a well-defined  polynomial in the variable $x$ of degree $n$, and the sequence
$(p_n(x))_{n \ge 0} $ is of binomial type:
\[
p_n(x+y) = \sum_{k \geq 0} \binom{n}{k} p_k (x) p_{n-k}(y).
\]

%

The above construction gives a family of polynomial sequences of binomial types with combinatorial
significance--they count the number of reluctant
functions in a binomial class. For example, let $\mathcal{B}$ contains all the reluctant functions for
which each block of the final preimage
is a singleton. Then $p_n(x)=x^n$.  Another example is that $\mathcal{B}$ contains all possible
reluctant functions. In this case
$p_n(x) = x(x+n)^{n-1}$, the Abel polynomial $x(x-an)^{n-1}$ with $ a=-1$.
This result can be proved by using Pr\"{u}fer codes, see e.g. \cite[Prop. 5.3.2]{Stan2}.
 More examples of binomial classes  are discussed in the next section.

For a sequence of numbers $(a_0, \ldots, a_{n-1})$, let $a_{(0)} \leq  \cdots \leq a_{(n)}$ be the non-decreasing
rearrangements of the terms $a_i$. The value of  $a_{(i)}$ is called the $i$-th order statistic of the sequence.
We will combine the notations of reluctant functions and order statistics.
From now on we assume $S=\{s_0, \ldots, s_{n-1}\}$, $x$ is a positive integer,  and $X=\{ 1,  \ldots, x\}$.
Associated with any reluctant function $f$ from $S$ to $X$  a sequence
$\vec{x}=(x_0, \ldots, x_{n-1}) \in X^n$ where $x_i$ is the final
image of $s_i$.

Assume that $z_0 \leq \cdots \leq z_{n-1}$ are integers in $X$. For a binomial class $\mathcal{B}$ of reluctant functions enumerated
by $p_n(x)$, define  $\mathcal{O}rd(z_0, \ldots, z_{n-1})$, \emph{the set of reluctant functions of length $n$
whose order statistics are bounded by $\ZZc$}, by letting
$$
\mathcal{O}rd(z_0, \ldots, z_{n-1}) = \{ f \in F(S, X):  x_{(0)} \leq z_0, \ldots, x_{(n-1)} \leq z_{n-1} \}.
$$
Then we have the following equation about $ord(z_0, \ldots, z_{i-1}) =
|\mathcal{O}rd(z_0, \ldots, z_{n-1}) |.$
\begin{theorem} \label{thm:OS}
With the notation as above, it holds that, for every $n \in \NNb$,
 \begin{eqnarray}\label{OS-recurrence}
    p_n(x) = \sum_{i=0}^n \binom{n}{i} p_{n-i}(x-z_i)  \cdot ord(z_0,  \ldots, z_{i-1}).
 \end{eqnarray}
\end{theorem}
\begin{proof}
 For any reluctant function $f \in F(S, X)$, let $\kappa(f)$ be the maximal index $i$ such that
 \begin{eqnarray} \label{condition}
x_{(0)} \leq z_0, \ldots, x_{(i-1)} \leq z_{i-1}.
\end{eqnarray}
The maximality of   $i$ implies that
$$
z_i < x_{(i)} \leq x_{(i+1)} \leq \cdots \leq x_{(n-1)} .
$$
Let $X_1$  be the subset of $X$ consisting of $\{1,  \ldots, z_{i} \}$, and
$X_2=X \setminus X_1=\{ z_i+1, \ldots, x\}$.  Assume
$(x_{r_0}, \ldots, x_{r_{i-1}})$  is the subsequence of $\vec{x}$
from which the sequence $(x_{(0)}, \ldots, x_{(i-1)})$  was obtained by rearrangement,
and let $A =\{s_{r_0}, \ldots, s_{r_{i-1}}\}  \subseteq S$.
Then the reluctant function $f$ is obtained by piecing two functions,
$f_A$ and $f_{S \setminus A}$ , the restrictions of $f$ on $A$ and $S\setminus A$,  together.
Since $\mathcal{B}$ is a binomial class, we have that
$f_A \in F(A, X_1), f_{S \setminus A} \in F(S \setminus A, X_2)$.
Furthermore, the function $f_A$ has the property that its order statistics are bounded
by $\ZZc$, and hence belongs to $\mathcal{O}rd(z_0, \ldots, z_{i-1})$.
The function $f_{S\setminus A}$ can be any reluctant function from
$S \setminus A$ to $X_2$.

Conversely, any pair of subsequences as described above can be reassembled into a reluctant
function from $S$ to $X$.
In other word, the decomposition $f \rightarrow (f_A, f_{S\setminus A})$ defines a bijection
from the set $F(S, X)$ to
\begin{eqnarray} \label{equ:union}
\bigcupdot_{i=0}^n  \bigcupdot_{A=\{r_0,  \ldots, r_{i-1}\}}
F_{ord}(A, \{1,\ldots, z_i\}) \otimes F(S\setminus A,  \{z_i+1,\ldots, x\} )
\end{eqnarray}
where $F_{ord}(A, \{1, \ldots, z_i\})$ is the set of reluctant functions from the set
$\{s_i: i \in A\}$  to $\{1, \ldots, z_i\}$  whose order statistics are bounded by $\ZZc$.

Now counting  the number of elements in the disjoint union of \eqref{equ:union}, we
get Eq.~\eqref{OS-recurrence}.
\end{proof}

Comparing Eq.~\eqref{OS-recurrence} with the linear recurrence of Proposition \ref{prop:recursive_relation},
we obtain a combinatorial interpretation of the generalized
\gonc basis.
\begin{theorem} \label{Order-statistics}  
Let $p_n(x)$ count the number of reluctant functions in a binomial class $\mathcal{B}$.
Assume $p_n(x)$ is the  sequence of basic polynomials of the delta operator $\dd$, and $t_n(x; \dd, \ZZc)$
is the $n$th generalized \gonc polynomial associated to the
pair $(\dd, \ZZc)$ with $\ZZc=(z_0, z_1, \ldots )$. Then
 \begin{eqnarray} \label{eq:order_sta}
  ord(z_0, \ldots, z_{n-1}) = t_n(x\fixed[0.15]{ \text{ }}; \dd, ( x-z_i)_{i \ge 0} )
 = t_n(0\fixed[0.15]{ \text{ }}; \dd,  -\ZZc ).
 \end{eqnarray}
 That is, $t_n(0\fixed[0.15]{ \text{ }}; \dd, -\ZZc)$ counts the number of reluctant functions of the binomial class
 $\mathcal{B}$ whose order statistics are bounded by $\ZZc$.
\end{theorem}

\section{Examples}
\label{sec:examples}
In this section we give some examples of sequences of polynomials of binomial type that enumerate binomial
classes. In each case, we describe the delta operator, the associated generalized \gonc polynomials, and their combinatorial
significance. We also compute $\dd$-Abel polynomials.

In  \cite{MuRo70}  Mullin and Rota introduced two important families of binomial classes of reluctant functions.
The first one is class $\mathcal{B}(T)$
where $T$ is a family of rooted trees. The class $\mathcal{B}(T)$ consists of all reluctant functions whose
final preimages are labeled rooted forests on
$S$ each of whose components is isomorphic to a rooted tree in the family $T$. Examples 1--5 belong to this family.
The second family of binomial classes is formed by taking a subclass of $\mathcal{B}(T)$: one only
allows those reluctant
functions in $\mathcal{B}(T)$ having the property that their final coimage coincides with their final preimage.
In other words, each rooted tree
in the final preimage is mapped to a distinct element in $X$. Such a subclass, denoted by $\mathcal{B}_m(T)$,
is called the \emph{monomorphic class associated to $\mathcal{B}(T)$} and is ultimately a generalization of the notion
of injective function. Example 6 and 7 belong to the monomorphic family.

The combinatorial interpretation of generalized \gonc polynomials are closely related to vector-parking functions.
Hence we recall the necessary notations on parking functions.
More  results and theories of parking functions can be found in the survey
paper \cite{Yan14}.
Let $\mathbf{u}=(u_i)_{i \geq 1}$ be a sequence of non-decreasing
positive integers. A $\mathbf{u}$-parking function of length $n$ is a sequence $(x_1, \ldots, x_n)$ of positive
integers whose order
statistics satisfy $x_{(i)} \leq u_i$. When $u_i=i$, we get the classical parking functions, which was originally
introduced by Konheim and Weiss \cite{KW66}.  Classical parking functions have a ``parking description'' as follows.
\begin{quotation}
There are $n$ cars $C_1, \ldots, C_n$ that want to park on
a one-way street with ordered parking spaces $1, \ldots, n$.
Each car $C_i$ has a preferred space $a_i$. The cars enter the street one at a time in the
order $C_1, \ldots, C_n$. A car tries to park in its preferred space. If that space
is occupied, then it parks in the next available space. If there is no space then the car leaves the
street (without parking). The sequence $(a_1, \ldots, a_n)$ is called a \emph{parking function  of length $n$}
if all the cars can park, i.e., no car leaves the street.
\end{quotation}
An equivalent definition for classical parking functions is that at least $i$ cars prefer the parking spaces
of labels $i$ or less.
Similarly a $\mathbf{u}$-parking function of length $n$, where $\mathbf{u} = (u_1, \ldots, u_n)$ is a vector of
positive integers, can be viewed as a parking preference sequence in which at least $i$ cars prefer the
parking spaces of labels $\le u_i$ (out of a total of $x \geq u_n$ parking spaces).

As explained in the following examples, the classical \gonc polynomial associated to the differential operator $D$ enumerate
$\mathbf{u}$-parking functions, while generalized \gonc polynomials associated to other delta operators
for binomial class $\mathcal{B}(T)$
enumerate variant forms of the parking scheme, in which the cars arrive in groups with certain special structures,
cars in the same group have the same preferred space, and there are at least $i$ cars preferring spaces of label $u_i $ or less.
Similarly, the generalized \gonc polynomials associated to other delta operators for the monomorphic
class $\mathcal{B}_m(T)$ enumerate those with the additional property that different groups have different preferences.
\subsection{The standard power polynomials}
\label{sec:standard}
The  sequence $(x^n)_{n \geq 0}$ is the basic polynomials of the differential operator $D$.
It enumerates the binomial class $\mathcal{B}(T_0)$, where $T_0$ consists of a single tree with only one vertex
(viz., an isolated vertex which is also the root of the tree).

A reluctant function in this class is just a usual function from $S$ to $X$, which can be represented by the sequence
$\vec{x}=(f(s_1),  \ldots, f(s_n))$.   \gonc polynomials $g_n(x\fixed[0.15]{ \text{ }}; \ZZc)$ associated to $(D, \ZZc)$ are  the
classical ones studied in
\cite{KuYan}, for which $(-1)^ng_n(0; \ZZc)$  enumerates the number of $\mathbf{z}$-parking functions
of length $n$, \cite[Theorem 5.4]{KuYan}.

If $\ZZc$ is the arithmetic progression $z_i=a+bi$,  then $g_n(x\fixed[0.15]{ \text{ }}; \ZZc)$ is  a shift of the  classical Abel polynomials, or more explicitly,
$$
 g_n(x\fixed[0.15]{ \text{ }}; (a+bi)_{i \ge 0})= (x-a)(x-a-nb)^{n-1}.
$$
It follows that  $P_n(a, a+b, \ldots, a+(n-1)b) = a(a+nb)^{n-1}$, where $P_n(z_0, \ldots, z_{n-1})$ is the number
of $\mathbf{z}$-parking
functions, i.e., positive integer sequences whose order statistics are weakly bounded by $\mathbf{z}$. In particular,
for $a=b=1$ we recover the formula for ordinary parking functions $P_n(1,\ldots, n)=(n+1)^{n-1}$.

\subsection{Abel polynomials}
\label{sec:abel-poly}
The Abel polynomial with parameter $a$, namely $A_n(x;a)=x(x-na)^{n-1}$, is the $n$-th basic polynomial  of the delta operator
$\dd = E_a D =D E_a$. When $a=-1$, $A_n(x;-1)=x(x+n)^{n-1}$. This polynomial counts the reluctant functions in the
binomial class $\mathcal{B}(T)$
where $T$ contains all possible rooted trees. In fact, such reluctant functions from $S$ to $X$ are represented as
rooted forests on vertex set $S \cup X$ with $X$ being the root set, whose number can be computed  by using Pr\"{u}fer codes.
We remark that the rooted forest corresponding to a reluctant function in $\mathcal{B}(T)$ is different from
the rooted forest $T_f$ in the final coimage of  $f$. The latter is a rooted
tree on $S$ only. The rooted forest corresponding to $f$ can be obtained from $T_f$ by
replacing every root $s_0$ of $T_f$ with an edge from $s_0$ to $f(s_0) \in X$, and letting every vertex of $X$ be a root.

Similarly, if $a=-k$ for a positive integer $k$, then $A_n(x; -k) =x(x+nk)^{n-1}$ enumerates the
reluctant functions in the binomial class $\mathcal{B}(T_k)$, where $T_k$ contains all the rooted $k$-trees,
which are rooted trees each of whose edge is colored by one of the colors $0, 1, \ldots, k-1$.
Such trees were studied in \cite{Stanley96, Yan97}.
In particular  the reluctant functions in $\mathcal{B}(T_k)$ can be represented as
sequences of rooted $k$-forests of length $x$,  which are defined in \cite{Yan97}  and  proved   to be  in bijection with
the  $(a,b)$-parking functions with $a=x$ and $b=k$.

By Proposition \ref{prop:shift_by_AP},  $t_n(x;E_{-k}D, \ZZc)$, the  generalized \gonc polynomial  associated
to $(E_{-k}D, \ZZc)$ is the same as $g_n(x; (z_i -ki)_{i\ge 0})$,
where $g_n(x; \ZZc)$ is the classical \gonc polynomial.

The polynomial $t_n(x; E_{-k}D, -\ZZc) = g_n(x; (-z_i-ki)_{i \ge 0})$, when evaluated at $0$,
has two combinatorial interpretations. On one hand, it
gives  the number of $\mathbf{u}$-parking functions with $\mathbf{u}=(u_i=ki+z_i)_{i \geq 0} $. On the other hand,
by Theorem \ref{Order-statistics} $t_n(0; E_{-k}D, -\ZZc)$  also counts the number of ways that
 $n$ cars form disjoint groups,  each group is equipped with a structure of rooted $k$-tree,  cars in the same group have the same
preference,  and the order statistics of the parking sequence are
bounded by $\ZZc$.

The $\dd$-Abel polynomial is of the form $t_n(x; E_{-k}D,  (a+bi)_{i \ge 0} ) = (x-a)(x-a-nb+nk)^{n-1}$.


\subsection{Laguerre  polynomials}
\label{sec:laguerre}
The $n$th Laguerre  polynomial $L_n(x)$ is given by the formula
\begin{eqnarray}
 L_n(x)= \sum_{k \geq 0} \frac{n!}{k!} \binom{n-1}{k-1} (-x)^k.
\end{eqnarray}
This is the $n$-th basic polynomial of the Laguerre delta operator $K:=D(D-I)^{-1} = -\sum_{i \ge 0} D^i$.

The coefficients $\frac{n!}{k!} \binom{n-1}{k-1}$ are called the (unsigned) Lah numbers,
which  are also the coefficients expressing rising factorials in terms of falling factorials.
See Example~\ref{sec:upper-fac}.

The polynomial $L_n(-x)$ enumerates the binomial class $\mathcal{B}(T_P)$, where $T_P$ is the set of all
rooted trees which is a path rooted at one of its leaves. To see this, consider all such reluctant functions
whose final preimage contains exactly $k$ rooted paths. To get such $k$
paths, we can linearly order all the elements of $S$ in a row and then cut it into $k$ nonempty segments, for
each segment let
the first element be the root. There are $n! \binom{n-1}{k-1}$ ways.
Since the paths are unordered, we divide $k!$ to get the number of sets of $k$ rooted paths.   Then multiplying
$x^k$ to get all functions from the set of $k$ paths  to $X$.

Let $t_n(x\fixed[0.15]{ \text{ }} ;K,  \ZZc)$ be the generalized \gonc polynomial associated to $(K, \ZZc)$. By Theorem \ref{thm:OS}
and its proof, we
get that the number of reluctant functions in  $\mathcal{B}(T_P)$ whose order statistics are bounded by $\ZZc$ is given by
$t_n(-x\fixed[0.15]{ \text{ }} ; K, -x+\ZZc)=t_n(0\fixed[0.15]{ \text{ }} ; K, \ZZc)$.  Equivalently, $t_n(0\fixed[0.15]{ \text{ }} ; K, \ZZc)$ counts the number of parking schemes
in which $n$ cars want to park in a parking lot of $x$ spaces such that (i)  the cars
arrive in disjoint groups,(ii) each group forms a queue,   (iii) all cars in the same queue prefer the same space,
and (iv) the order statistics of the preference sequence is bounded by $\ZZc$.

The $n$-th $\dd$-Abel polynomial associated with the operator $\dd = K$ and the arithmetic grid $\ZZc = (a+bi)_{i \ge 0}$ is given by
$$
 t_n(x\fixed[0.15]{ \text{ }} ; K,  (a+bi)_{i \ge 0} ) = (a-x) \sum_{k=1}^n \frac{n!}{k!} \binom{n-1}{k-1}(a+nb-x)^{k-1}.
$$
In particular, $t_n(0\fixed[0.15]{ \text{ }}; K,  (a+bi)_{i \ge 0} ) = a \sum_{k=1}^n  \frac{n!}{k!} \binom{n-1}{k-1}(a+nb)^{k-1}$, which,
for $a=b=1$, supplies the sequence
$1, 5, 46, 629, 11496, \ldots$,  namely A052873 in the On-Line Encyclopedia of Integer
Sequences (OEIS) \cite{OEIS}, where one can find the exponential generating function and an asymptotic formula.

\subsection{Inverse of the Abel polynomial $A_n(x;-1)$} \ \    \label{sec:inv_abel_polys}
Let $p_n(x) := \sum_{k\geq 0} \binom{n}{k} k^{n-k} x^k$.
Then $(p_n(x))_{n \ge 0}$ is a sequence of binomial type: This is actually the basic sequence of the delta operator
$\dd $ whose $D$-indicator is the compositional inverse of $f(t)=te^t$ (often referred to as the Lambert $w$-function),
and it is also the inverse sequence of the Abel polynomials $(A_n(x;-1))_{n \geq 0}$ under the umbral composition.

In fact, the sequence $(p_n(x))_{n \ge 0}$ enumerates the binomial class $\mathcal{B}(T_1)$, where $T_1$ contains all
the rooted trees with depth at most $1$, i.e., stars $\{S_k: \ k \geq 0 \}$ where
$S_k$ is the tree on vertices $\{v_0, \ldots, v_k\}$ with root $v_0$ and edges $\{ v_0, v_i\}$ for $i=1, \ldots, k$, \cite[Sec. 7]{MuRo70}.

By Theorem \ref{Order-statistics}, the generalized \gonc polynomial associated to $(\dd, \ZZc)$ gives a formula for
the number of parking schemes  such that the cars arrives in disjoint groups, each group has a leader,
all cars in the same group prefer the same space, and the order statistics of the preference sequence is bounded by $\ZZc$.


In addition, it follows from the above and Theorem \ref{thm:delta_Abel} that the $n$-th $\dd$-Abel polynomial associated wit
h the arithmetic grid $\ZZc = (a+bi)_{i \ge 0}$ is given by
$$
 t_n(x; \dd,  (a+bi)_{i \ge 0} ) = (x-a) \sum_{k=1}^n \binom{n}{k} k^{n-k} (x-a-nb)^{k-1}.
$$
In particular, $t_n(0; \dd, (-a-bi)_{i \ge 0})=a\sum_{k=1}^n  \binom{n}{k} k^{n-k} (a+nb)^{k-1}$;
for $a=b=1$, this yields the sequence $1, 5, 43, 549, 9341, \ldots$, which is  A162695 in OEIS, where one can
find the exponential generating function and an asymptotic formula.

\subsection{Exponential  polynomials}
\label{sec:expo_polys}

The $n$-th exponential polynomial, also called the \emph{Touchard polynomial} or the \emph{Bell polynomial},  is given by
$b_n(x)=\sum_{k=1}^n S(n,k) \fixed[0.15]{ \text{ }} x^k$,
where the coefficients $S(n,k)$ are the familiar Stirling numbers of the second kind,
see \cite[pp.~747--750]{RoKaOd73}.

Introduced by J. F. Steffensen in his 1927 treatise on interpolation \cite{Stef} and later reconsidered,
in particular, by J.~Touchard \cite{Touc} for their combinatorial and arithmetic properties, the exponential
polynomials are the basic polynomials of the delta  operator
\begin{equation}
\label{equ:touchard_operator}
\mathfrak{b} := \log(I + D) := \sum_{i \ge 1} (-1)^{i+1} \frac{1}{i} \iter{D}{i}.
\end{equation}
%
The exponential polynomials also enumerate a binomial class $\mathcal{B}(T)$ of reluctant functions from $S$ to $X$,
where we require that
elements in $S$ are totally ordered, i.e., there is an order such that $s_1 < \cdots < s_n$. Now let $T$ be
the family of
rooted paths labeled by $S$ such that the labels are monotone along the path, with the root having the largest label.
Correspondingly, the generalized \gonc polynomial $t_n(x\fixed[0.15]{ \text{ }} ; \mathfrak{b}, \ZZc)$ gives the
enumeration of parking schemes
in which $n$ cars arrive in disjoint groups,  all cars in the same group prefer the same space, and the order statistics of the
preference sequence are bounded by  $\ZZc$.
%

Thus, we get from the above and Theorem \ref{thm:delta_Abel} that the $n$-th $\dd$-Abel polynomial associated with the arithmetic grid $\ZZc = (a+bi)_{i \ge 0}$ is
$$
 t_n(x\fixed[0.15]{ \text{ }};  \mathfrak{b}, (a+bi)_{i \ge 0}) = (x-a) \sum_{k=1}^n S(n,k) (x-a-nb)^{k-1}.
$$
In particular,
$t_n(0\fixed[0.15]{ \text{ }}; \mathfrak{b}, (-a-bi)_{i \ge 0})=a\sum_{k=1}^n  S(n,k) (a+nb)^{k-1}$; for $a=b=1$, this gives the sequence $1, 4, 29, 311, 4447 \ldots$,  which, after a shift of index,  is  A030019 in OEIS.
The sequence also has other combinatorial interpretations,  for example, as  hypertrees
on $n$ labeled vertices \cite{OEIS}.   It would be interesting to find bijections between the parking sequences and
the hypertrees.

\vspace{.3cm}

The next two examples correspond to monomorphic classes associated to some binomial class $\mathcal{B}(T)$. As pointed out in
\cite{MuRo70},      if the reluctant functions in  $\mathcal{B}(T)$ are counted by $p_n(x)$ where
$$
 p_n(x) = \sum_{k=0}^n a_{k}x^k,
$$
then the basic sequence counting $\mathcal{B}_m(T)$ is
$$
 q_n(x)= \sum_{k=0}^n a_{k} x_{(k)}.
$$

\subsection{Lower factorial  polynomials} \label{sec:lower-fac}
The lower factorial $x_{(n)}=\prod_{i=0}^{n-1} (x-i)$ is the basic polynomial for the monomorphic class $\mathcal{B}_m(T_0)$, where
$T_0$ is as described in Example \ref{sec:standard}. The corresponding delta operator is the forward difference operator
$\Delta_{1,0}=E_1-E_0=E_1-I$. It counts the number of one-to-one functions from $S$ to $X$.

By Theorem  \ref{thm:OS}, the generalized \gonc polynomials $t_n(x)$ associated to $(\Delta_{1,0}, \ZZc)$ give the number of
one-to-one functions from $S$ to $X$ whose order statistics are bounded by $\ZZc$ via Eq.~\eqref{eq:order_sta}.
Assume $x \geq n$.
Note that any sequence $1\leq x_1< \cdots < x_n \leq x$ can be  represented geometrically as a strictly increasing lattice
path in the plane from $(0,0)$ to $(x-1, n)$ using only steps $E=(1,0)$ and $N=(0,1)$: one simply takes the $N$-steps from
$(x_i-1, i-1)$ to $(x_i-1, i)$, and connects the $N$-steps with $E$-steps, (see Figure~\ref{lattice-path}).

\begin{figure}[ht]
 \begin{center}
\begin{tikzpicture}[scale=0.5]
\draw [very thin] (0,0)--(8,0) (0,1)--(8,1) (0,2)--(8,2) (0,3)--(8,3) (0,4)--(8,4)
(0,0)--(0,4) (1,0)--(1,4) (2,0)--(2,4) (3,0)--(3,4) (4,0)--(4,4) (5,0)--(5,4) (6,0)--(6,4) (7,0)--(7,4) (8,0)--(8,4);
\draw [ultra thick,blue] (0,0)--(0,1)--(2,1)--(2,2)--(3,2)--(3,3)--(6,3)--(6,4)--(7,4);
\node at (1,0) {$\times$};  \node at (3,1) {$\times$};  \node at (5,2) {$\times$};  \node at (7,3) {$\times$};

\end{tikzpicture}
\caption{Lattice path corresponding to the sequence $(1, 3,4,7)$ with $x=8$. The stars indicate the right boundary
$(1, 3, 5, 7)$.  }  \label{lattice-path}
\end{center}
\end{figure}
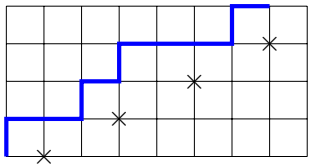
Clearly there are $n!$  one-to-one functions from $S$ to $X$ whose images are $\{x_1, \ldots, x_n\}$.
Thus Theorem \ref{Order-statistics}
implies that $\frac{1}{n!} t_n(0\fixed[0.15]{ \text{ }}; \Delta_{1,0}, - \ZZc)$ counts the number of strictly increasing lattice paths from
$(0,0)$ to $(x-1, n)$ with strict right boundary $(z_0, \ldots, z_{n-1})$, viz. with right boundary $(z_0, \ldots, z_{n-1})$ that never touch the points $\{(z_i, i): 0 \leq i < n\}$.

In particular, we have $t_n(x\fixed[0.15]{ \text{ }}; \Delta_{1,0}, (a+bi)_{i \ge 0} ) = (x-a) (x-a-nb-1)_{(n-1)}$, so the number of
 strictly increasing lattice paths from $(0,0)$ to $(x-1, n)$ with the affine right boundary $(a, a+b, a+2b, \ldots)$, where
 $a>0$ and $ b\geq 0$ are  integers,
 is given by
 $$
  \frac{1}{n!} a(a+nb-1)_{(n-1)} = \frac{a}{a+nb} \binom{a+nb}{n}.
 $$
When $a=b=1$, the above number is $\frac{1}{n+1} \binom{n+1}{n}=1$, since there is only one strictly increasing lattice path from
$(0,0)$ to $(x-1, n)$ that is bounded by $(1, \ldots, n)$.

\subsection{Upper factorial  polynomials} \label{sec:upper-fac}
The upper  factorial $x^{(n)}=(x+n-1)_{(n)}= \prod_{i=0}^{n-1} (x+i)$ is related to the Laguerre polynomials by the equation
$$
 x^{(n)}= \sum_{k \geq 0} \frac{n!}{k!} \binom{n-1}{k-1} x_{(k)}.
$$
Hence $x^{(n)}$  is the basic polynomial for the monomorphic
class $\mathcal{B}_m(T_p)$ where $T_p$ is as described in Example \ref{sec:laguerre}.
The corresponding delta operator for $(x^{(n)})_{n\geq 0}$ is the backward difference operator
$\Delta_{0,-1}=E_0-E_{-1} = I-E_{-1}$.

Monomorphic reluctant functions in $\mathcal{B}_m(T_p)$ can also be described  by lattice paths. Assume the final preimages of a
reluctant function in $\mathcal{B}_m(T_p)$ consists of $k$ paths of lengths $p_1, \ldots, p_k$ whose images are
$x_1 < \cdots < x_k$.  It corresponds to a lattice path from $(0,0)$ to $(x-1, n)$ whose consecutive vertical runs are
given by $p_1$ $N$-steps at $y=x_1-1$, followed by $p_2$ $N$-steps at $y=x_2-1$, and so on. See Figure~\ref{lattice-path2}
for an example.

\begin{figure}[ht]
 \begin{center}
\begin{tikzpicture}[scale=0.5]
\draw [very thin] (0,0)--(6,0) (0,1)--(6,1) (0,2)--(6,2) (0,3)--(6,3) (0,4)--(6,4) (0,5)--(6,5) (0,6)--(6,6) (0,7)--(6,7) (0,8)--(6,8)
(0,0)--(0,8) (1,0)--(1,8) (2,0)--(2,8) (3,0)--(3,8) (4,0)--(4,8) (5,0)--(5,8) (6,0)--(6,8);
\draw [ultra thick,blue] (0,0)--(0,3)--(1,3)--(1,5)--(3,5)--(3,7)--(5,7)--(5,8);
\node at (1,0) {$\times$};  \node at (1,1) {$\times$};  \node at (2,2) {$\times$};  \node at (2,3) {$\times$};
\node at (4,4) {$\times$}; \node at (4,5) {$\times$}; \node at (6,6) {$\times$}; \node at (6,7) {$\times$};
\end{tikzpicture}
\caption{Lattice path corresponding to the images  $(1,1,1,2,2,4,4,6)$ with $x=6$. The stars indicate the right boundary
$(1, 1,2,2,4,4,6,6)$.  }  \label{lattice-path2}
\end{center}
\end{figure}
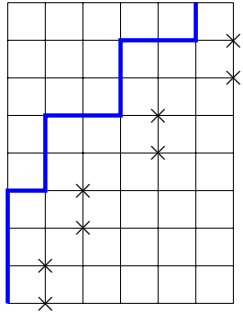

The  labels on each path  can be recorded by labeling the $N$-steps of the lattice path. Again there are $n!$ many labels possible for
each lattice path. Hence by Theorem \ref{thm:OS} we know that if
$t_n(x; \Delta_{0,-1}, \ZZc)$ is the generalized \gonc polynomial
 associated to the pair
$(\Delta_{0,-1}, \ZZc)$, then $\frac{1}{n!} t_n(0; \Delta_{0, -1}, -\ZZc)$ is  the number of lattice paths with the right boundary $\ZZc$,
a result first
established in \cite{KSY}.

In particular, $t_n(x; \Delta_{0,-1}, (a+bi)_{i \ge 0}) = (x-a) (x-a-nb+1)^{(n-1)}$, hence the number of
  lattice paths from $(0,0)$ to $(x-1, n)$ with strict affine right boundary $(a, a+b, a+2b, \ldots)$ for some  integers
 $a>0$ and $ b\ge 0$ is
 \begin{eqnarray} \label{for:catalan}
 \frac{1}{n!}  a(a+nb+1)^{(n-1)} = \frac{a}{a+n(b+1)} \binom{a+n(b+1)}{n},
 \end{eqnarray}
 a well-known result, see e.g. \cite[p. 9]{Mohanty}.
In particular,  for $a=1$ and $b=k$ for some positive integer $k$, it counts the number of lattice paths from the origin to $(kn, n)$ that never pass below
the line $x=yk$. In this case, \eqref{for:catalan} gives
$\frac{1}{1+(k+1)n} \binom{1+(k+1)n}{n}= \frac{1}{1+kn}\binom{(k+1)n}{n}$,
which is the $n$-th $k$-Fuss-Catalan number.
\section{Further remarks}
\label{sec:remarks}
In the theory of binomial enumeration, it is not really necessary to restrict oneself to reluctant functions, and we can in fact consider a more general
setting, as outlined in \cite[Sec. 2]{Ray88}. Following Joyal \cite{Joyal}, a \emph{species} $B$ is a covariant endofunctor on the
category of finite sets and bijections. Given a finite set $E$, an element $s \in B(E)$ is called a $B$-structure on $E$. A
\emph{$k$-assembly of $B$-structures on $E$} is then a partition $\pi$ of the set $E$ into $k$ blocks such that each block of $\pi$ is
endowed with a $B$-structure. Let $B_k(E)$ denote the set of all such $k$-assemblies.
For example, when $B$ is a set of rooted trees, a $k$-assembly of $B$-structures on $E$ is a $k$-forest of rooted trees with vertex set $E$.
But we can also take $B$ to be other structures, such as permutations, graphs, posets, etc.

To enumerate the number of assemblies of $B$-structures, we define sequences of nonnegative integers by
$$
 b_{n,k} = \left\{
 \begin{array}{ll}
   |B_k([n])|,  &  k \leq n, \\
   0, & k >n.
 \end{array}
\right.
$$
Mullin and Rota's work \cite{MuRo70} establishes that if $b_n(x)=\sum b_{n,k} x^k$ is the enumerator for assemblies of $B$-structures
on $[n]$, then $(b_n(x))_{n \geq 0}$ is a polynomial sequence of $\KK[x]$ of binomial type. Now we can interpret
the factor $x^k$ in $b_n(x)$ by considering all functions
(or monomorphic functions if one replaces $x^k$ with $x_{(k)}$) from the blocks of a $k$-assembly to a set $X$ of size $x$.  When $X$ is totally ordered, i.e., $X$
is isomorphic to the poset $[s] $ with numerical order, where $s=|X|$, we can consider all such functions
whose order statistics are bounded
by a given sequence. The enumeration for such $k$-assemblies with an order-statistics constraint is captured by the
associated generalized \gonc polynomials.

In principle, the above combinatorial description only applies to  polynomial sequences $(p_n(x))_{n \ge 0}$ of
binomial type with nonnegative integer coefficients. Mullin and Rota hinted at a generalization of their theory to
include polynomials with negative coefficients, and Ray \cite{Ray88} developed a  concept of  weight  functions  on  the  partition
category  which  allows  one  to  realize  any
binomial sequence, over any commutative ring with identity, as the enumerator of weights. It would be interesting
to investigate the role of generalized \gonc polynomials in such weighted counting, as well as in
other dissecting schemes as described in Henle \cite{Henle75},
and to find connections to rook polynomials, order invariants, Tutte invariants of combinatorial geometries, and symmetric functions.

\section*{Acknowledgments}
This publication was made possible by NPRP grant No. [5-101-1-025] from the Qatar National
Research Fund (a member of Qatar Foundation). The statements made herein are
solely the responsibility of the authors.

We thank Professor Graham for his continual support and encouragement throughout our research.

\end{document}